\documentclass{article}
\usepackage{latexsym}
\usepackage[utf8]{inputenc}
\usepackage{amsmath, amsfonts, amssymb, amsthm}
\usepackage{graphicx}
\usepackage{color}
\newtheorem{thm}{Theorem}
\newtheorem{defn}{Definition}
\newtheorem{lemma}{Lemma}

\newtheorem{rem}{Remark}
\newtheorem{cor}{Corollary}

\newcommand{\C}{\mathbb C}

\newcommand{\D}{\mathbb D}

\newcommand{\wto}{\stackrel{\mathrm{w}^*}{\to}}
\newcommand{\Om}{\Omega}
\newcommand{\absv}[1]{\left\lvert{#1}\right\rvert}
\newcommand{\Supp}{\ensuremath{\operatorname{supp}}}
\begin{document}
\title{Value Distributions of Derivatives of $K$-regular Polynomial Families}
\author{%
  Henriksen, Christian \\
  \texttt{chrh@dtu.dk} \\
  Department of Applied Mathematics and Computer Science \\
  Technical University of Denmark
  \and
  Petersen, Carsten Lunde \\
  \texttt{lunde@math.ku.dk} \\
  Department of Mathematical Sciences \\
  University of Copenhagen
  \and
  Uhre, Eva \\
  \texttt{euhre@ruc.dk} \\
  Department of Science and Environment \\
  Roskilde University
}
\date{\today}
\maketitle
\begin{abstract}
  Let $\Omega \in \C$ be a domain such that 
  $K:= \C \setminus \Omega$ is compact and non-polar.
  Let $(q_k)_{k>0}$ be a sequence of polynomials
  with $n_k$, the degree of $q_k$ satisfying $n_k \to \infty$,
  and let $(q_k^{(m)})_k$ denote the sequence of $m$-th derivatives.
  We provide conditions, which ensure that 
  the preimages $(q_k^{(m)})^{-1}(\{a\})$ uniformly equidistribute on $\partial \Omega$,
  as $k \to \infty$, for every $a \in \C$ and every $m = 0, 1, \ldots$
\end{abstract}
\noindent {\em \small 2020 Mathematics Subject Classification: Primary: 42C05, Secondary: 37F10, 31A15}
\noindent {\em \small Keywords: Orthogonal Polynomials, Julia set, Green's function}

\section{Introduction and main results} \label{intro}
In recent papers, Okuyama and Vigny showed that the asymptotic
value distribution of a sequence of $m$-th derivatives of iterates
of a polynomial, is generically the same as that of the
sequence of iterates itself,
see \cite{Okuyama1}, \cite{OkuyamaVigny}.
It raises the question, whether this behavior is special
for sequences of iterates, or if this could be true
in more general families of polynomials.
This question was answered affirmatively by Totik in \cite{Totik}, 
where he provides conditions on a sequence of polynomials $(q_k)_k$, 
which ensure that the sequence of derivatives $(q'_k)_k$ 
has the same asymptotic zero distribution as the sequence of polynomials itself. 
Furthermore, in \cite{HPU1} we have complemented Totik's result 
by providing different conditions, leading to the same conclusion.
 
In the current paper, we continue this line of work by exploring 
the important special case where the asymptotic
measure is the equilibrium measure of a compact set $K$.
This allows for stronger results,
which we will state in this section.

Throughout this paper we employ the following
conventions.
We let $K \subset \C$ be a compact, non-polar, polynomially convex set. 
Then $\Omega := \C \setminus K$ is a domain admitting 
a Green's function $g_\Omega$ with pole at $\infty$. 
Let $J:=\partial K=\partial \Omega$ and let $\omega$ 
denote the equilibrium measure for $K$.

We study sequences $(q_k)_{k>0}$ of polynomials, where
\begin{equation}
q_k(z)=\gamma_k z^{n_k}+O\left(z^{n_k-1}\right), n_k > 0
\end{equation}
with degree $n_k \to \infty$.

Given a non-constant polynomial $q$, we can write it as
\[ q(z) = \gamma \prod_{j=1}^n (z-z_j) \]
where $\gamma$ is the leading coefficient and $(z_j)$ are
the roots, repeated with respect to multiplicity.
We define the \emph{root distribution} of $q$ to be
the (Borel) probability measure $\mu_q$ defined by 
\[
  \mu_q = \frac{1}{n} \sum_{j=1}^n \delta_{z_j},
\]
where $\delta_z$ denotes the Dirac point mass at $z$.
Then for any set $U \subset \C$, the integer
$n\mu(U)$ equals the number of roots of $q$ in $U$
counted with multiplicity.

To alleviate notation, we denote the root distribution of $q_k$
by $\mu_k := \mu_{q_k}$.

For many important families studied in the literature,
the number of roots of $q_k$ in a closed set $L$ not meeting $K$
can be uniformly bounded above (see \cite[Example 6 and 7]{HPU1}
for two important examples).
In view of this we define
\begin{defn}[Centering]
  We say that $(q_k)_{k>0}$ \emph{centers} on
  the compact set $C$ if there exists $N$ such that the following two
  conditions hold:
  \begin{itemize}
    \item[1.]
      There exists $R > 0$ such that for all $k\geq N$
      \begin{equation}\label{centering1}
        \nonumber q_k^{-1}(\{0\})\subset \D(R).
      \end{equation}
    \item[2.] For every closed set $L$ with $L\cap C=\emptyset$ there exists
      $M = M(L)>0$ such that
      \[n_k \mu_k(L) \leq M\; \text{for all }k\geq N.\]
  \end{itemize}
  We will say that $(q_k)_{k>0}$ \emph{weakly centers} on the compact
  set $C$ if condition 2{.}
  is replaced with the weaker condition
  \begin{itemize}
    \item[2.](Weak) For every closed set $L$ with $L\cap C=\emptyset$,
      $\mu_k(L)=o(1)$ as $k\to \infty$.
  \end{itemize}
\end{defn}
Centering on $K$, which by our standing assumption is compact
and polynomially convex, is inherited by the sequence of derivatives,
see \cite[corollary 3]{HPU1}.
We include this result here for completeness.
\begin{thm}\label{thmInheritCenter}
  Let $K \subset \C$ be compact and polynomially convex.
  If a sequence of polynomials $(q_k)_{k>0}$ centers on $K$, then so
  does the sequence of its derivatives $(q'_k)_{k>0}$.
\end{thm}
We recently studied the root distribution of the sequence of
$m$th derivatives $(q^{(m)}_k)_{k>0}$ of a family of polynomials $(q_k)_k$ centering on 
a compact set, see \cite{HPU1}.
Adding control over the asymptotics of $q_k$ near infinity,
as $k \to \infty$ allows us to obtain stronger results. 
The control alluded to, is the following.
\begin{defn}[$K$-regularity]
  Let $K \subset \C$ be compact.
  We say that a sequence of polynomials $(q_k)_{k>0}$ is
  \emph{$K$-regular} if, for some $R > 0$,
  \begin{equation*}%\label{eq:genreg}
    \lim_{n\to\infty}{\textstyle\frac{1}{n_k}}\log|q_k(z)| = g_\Om(z)
  \end{equation*}
  uniformly on $\C\setminus \D(0, R)$.
\end{defn}
Notice that if the root locus of $q_k$ is uniformly bounded,
then uniform convergence on $\C\setminus \D(0, R)$ is equivalent
to local uniform convergence on $\C\setminus \D(0, R)$.
Also notice, that our definition of $K$-regularity is a generalization 
of $n$th-root regularity as defined in
\cite[Definition 3.1.2]{StahlTotik}.
We have that $K$-regularity is inherited by the sequence of derivatives.
\begin{thm}\label{thmInheritKreg}
 Let $K \subset \C$ be a compact, non-polar and polynomially convex set.
 If $(q_k)_k$ is $K$-regular, then so is $(q_k')_k$.
\end{thm}
When $\mu_k$ converges weak star to $\mu$, we write $\mu_k \wto \mu$.
The following theorem establishes under which conditions
$\mu_k$ converges weak star to the equilibrium measure $\omega$ on $K$. 
This is a generalization of Brolins Theorem \cite[Thm 16.1]{Brolin}.
\begin{thm}\label{thmBweakstarconv}
  Let $K \subset \C$ be a compact, non-polar and polynomially convex set 
  and let $\omega$ be the equilibrium measure on $K$.
  Suppose that the sequence of polynomials $(q_k)_{k>0}$ is $K$-regular.
  Then $\mu_k\wto \omega$ if and only if
  $(q_k)_{k>0}$ weakly centers on $J := \partial K$.
\end{thm}
(Weakly) centering on $J$ is not in general inherited by the sequence
of derivatives unlike centering on $K$.
It is clear that centering on $K$ is the same as centering
on $J$, when $K$ has no interior.
The next theorem shows that there is a more general class of compact sets 
(than those with empty interior), 
where weakly centering on $K$ implies weakly centering on $J$, 
when $(q_k)_k$ is $K$-regular.
\begin{thm}\label{thmNoinside}
  Let $K \subset \C$ be a compact, non-polar and polynomially convex set. 
  Suppose the equilibrium measure $\omega$ of $K$ satisfies
  $\omega(\overline{V}) = 0$ for every component $V$ of the interior of $K$.
  If $(q_k)_k$ is $K$-regular and weakly centers on $K$, 
  then $(q_k)_k$ weakly centers on $J:=\partial K$.
\end{thm}
The following theorem shows, that centering on $K$ is fairly robust
for $K$-regular sequences.
Define $\log^+$ by $\log^+(w) = \max\{0, \log(w)\}$.
\begin{thm}\label{thmPerturb}
  Let $K \subset \C$ be a compact, non-polar and polynomially convex set 
  and let $(a_k)_k\subset \C$ be a sequence 
  with $\log^+ \absv{a_k} = o(n_k)$ as $k\to\infty$.
  If $(q_k)_{k>0}$ centers on $K$ and is $K$-regular,
  then $(q_k-a_k)_k$ centers on $K$ and is $K$-regular. 
\end{thm}

The above theorem extends our previous results on zero distributions,
to value distributions, since zeros of $q_k - a_k$ are clearly just
the preimages of $a_k$ under $q_k$.

Combining the previous theorems, we get
\begin{thm}\label{thmCombined}
  Let $K \subset \C$ be a compact, non-polar and polynomially convex set, 
  and let $\omega$ be the equilibrium measure on $K$. 
  Suppose $K$ has the property that
  $\omega(\overline{V}) = 0$ for every component $V$ of the interior of $K$.
  Let $(a_k)_k\subset\C$ satisfy %be a sequence of complex numbers, with
  $\log^+ \absv{a_k} = o(n_k)$.
  If $(q_k)_k$ is $K$-regular and centers on $K$,
  then for every $m = 0, 1, \ldots$ the sequence of perturbed $m$th derivatives
  $(q_k^{(m)} - a_k)_k$ is $K$-regular, centers on $K$, 
  and has root distributions converging weak star to $\omega$.
\end{thm}
This result generalizes \cite[Theorem 1]{OkuyamaVigny}, 
to the vastly larger non-dynamical setting,
at the expense of the condition that $\omega(\overline{V}) = 0$ 
 for every component $V$ of the interior of $K$.
 
\section{Proofs of the results}
The potential of a finite Borel measure $\mu$ is given by
\[
  p_\mu(z) = \int \log \absv{z-w}\, d\mu(w),
\]
where we have used the sign convention of \cite{Ransford}.

The Green's function $g_\Omega$ with pole at infinity is 
$g_\Omega = p_\omega - I(\omega)$,
where $I(\omega)$ is the energy of the equilibrium measure for $K$, see \cite{Ransford}.
In particular we have
\[
  g_\Omega(z) = \log\absv{z} - I(\omega) + o(1), \text{ as }z\to \infty.
\]

We start with the proof of Theorem \ref{thmInheritKreg}.
\begin{proof}
  If $q$ is a non-constant polynomial, it is well known that 
  the roots of $q'$ lie in the convex hull of the zero locus of $q$.
  In particular the zero locus of $q_k'$ is contained in $\D(0,R)$ for $k$
  sufficiently big.
  We can assume $K \subset \D(0,R)$ by possibly increasing $R$.

  Take $R_2 > R$.
  We will show that
  $\frac{1}{n_k-1} \log \absv{q_k'}
    = g_\Omega + o(1)$
  uniformly outside $\D(0, R_2)$.
  Since for each $k$,
  $\frac{1}{n_k-1} \log \absv{q_k'(z)} = \log \absv{z} + O(\frac{1}{\lvert z \rvert})$,
  as $z \to \infty$, and $g_\Omega(z) = \log \absv{z} + O(\frac{1}{\lvert z \rvert})$ as $z \to \infty$, the function
  $h_k = \frac{1}{n_k-1} \log \absv{q_k'}-g_\Omega$
  is harmonic outside $\D(0, R)$ for sufficiently large $k$ and
  extends harmonically to $\infty$, by letting
  $h_k(\infty) = \lim_{z\to\infty} h_k(z) < \infty$.
  By the maximum principle for harmonic functions, it is enough to show
  that $h_k \to 0$ uniformly on $\partial \D(0, R_2)$ as
  $k \to \infty$.
  I.e.\ if $\frac{1}{n_k-1} \log \absv{q_k'} \to g_\Omega$ uniformly on
  $\partial \D(0, R_2)$, then
  $\frac{1}{n_k-1} \log \absv{q_k'} \to g_\Omega$
  uniformly on $\C \setminus \D(0, R_2)$.

  Since $g_\Omega$ is harmonic on $\Omega$,
  $g_\Omega' := 2 \frac{\partial}{\partial z} g_\Omega$
  is a holomorphic function.
  Also, since $g_\Omega = p_\omega - I(\omega)$,
  $g_\Omega$ has no critical points for $\absv{z} \geq R$.

   Taking $2 \frac{\partial}{\partial z}$ derivatives on both sides of
  \[
    \lim_{n\to\infty}{\textstyle\frac{1}{n_k}}\log|q_k(z)| = g_\Om(z),
  \]
  we get by local uniform convergence
  \[
    \frac{1}{n_k} \frac{q_k'(z)}{q_k(z)} = g'_\Omega(z) + o(1),
  \]
  where $o(1) \to 0$ as $k \to \infty$ locally uniformly on $\absv{z} \geq R$.
  Hence $q_k'(z) = (g_\Omega'(z) + o(1))n_k q_k$, so
  \begin{align*}
    \frac{1}{n_k-1} \log \absv{q_k'(z)}
    = & \phantom{+} \frac{1}{n_k-1} \log(n_k)              \\
      & + \frac{1}{n_k-1} \log(\absv{g_\Omega'(z) + o(1)}) \\
      & + \frac{1}{n_k-1} \log \absv{q_k(z)}.
  \end{align*}
  Now as $k\to \infty$,
  $\frac{1}{n_k-1} \log(n_k) = o(1)$.
  Also $g_\Omega'(z) \neq 0$ when $\absv{z} \geq R$, so
  $\frac{1}{n_k-1} \log \absv{g_\Omega'(z) + o(1)} = o(1)$.
  Finally,
  $\frac{1}{n_k-1} \log \absv{q_k(z)} = g_\Omega(z) + o(1)$.
  Here the $o(1)$-term in each of the three expression is uniform
  on $\partial\D(0, R_2)$.
  So we get
  \[
    \frac{1}{n_k-1} \log \absv{q_k'(z)} \to g_\Omega(z),
  \]
  uniformly on $\partial\D(0, R_2)$, finishing the proof.
\end{proof}

We will use the following lemma to prove Theorem \ref{thmBweakstarconv}.

\begin{lemma}\label{lem:potargument}
  Suppose $\mu$ and $\nu$ are finite Borel measures supported on $J$.
  If $p_\mu(z) = p_\nu(z)$, outside $\D(0,R)$ for some $R$,
  then $\mu = \nu$.
\end{lemma}

\begin{proof}
  Since both $\mu$ and $\nu$ are supported on $J$, the potentials
  $p_\mu$ and $p_\nu$ are harmonic on $\Omega$.
  Since $p_\mu$ and $p_\nu$ agree on $\lvert z \rvert > R$, they
  agree on
  all of $\Omega$, by the identity principle of harmonic functions.

  Since $\Omega$ is connected, it follows from \cite[Theorem
    3.8.3]{Ransford},
  that for every $z \in \partial \Omega = J$, we have
  \[
    p_\mu(z) 
    = \limsup_{\substack{\zeta \in \Omega\\ \zeta \to z}} p_\mu(\zeta)
    = \limsup_{\substack{\zeta \in \Omega\\ \zeta \to z}} p_\nu(\zeta)
    = p_\nu(z).
  \]
  So $p_\mu = p_\nu$ on $J$.

  We know that $p_\mu = p_\nu$ except possibly on components of the
  interior
  of $K$.
  Towards a contradiction suppose that there exists a component $U$
  of the interior of $K$,
  on which $p_\mu \neq p_\nu$.
  We have already shown $p_\mu = p_\nu$ on $\partial U \subset J$.
  Also, $p_\mu$ and $p_\nu$ are harmonic on $U$.
  We can assume that there exist $z_0 \in U$
  and $\epsilon > 0$ so that $p_\mu(z_0) > p_\nu(z_0) + \epsilon$,
  possibly by exchanging the roles of $p_\mu$ and $p_\nu$.
  Let $V$ be a connected component of the set
  $\{z \in U: p_\mu(z) > p_\nu(z) + \epsilon\}$.
  Applying the maximum principle to $p_\mu - p_\nu$,
  we see that $V$ is not compactly contained in $U$.
  Let $w \in \partial U \cap \partial V$.
  Then, again by \cite[Theorem 3.8.3]{Ransford},
  \[
    p_\mu(w)
    = \limsup_{\substack{\zeta \in V\\ \zeta \to w}} p_\mu(\zeta)
    \geq \limsup_{\substack{\zeta \in V\\ \zeta \to w}} p_\nu(\zeta)
       + \epsilon
    = p_\nu(w) + \epsilon.
  \]
  contradicting $p_\mu(w) = p_\nu(w)$.

  We have shown that $p_\mu = p_\nu$ on $\C$.
  By uniqueness (see \cite[Corollary 3.7.5]{Ransford}), $\mu = \nu$.
\end{proof}

We can now prove Theorem \ref{thmBweakstarconv}.

\begin{proof}
  The ``only if'' part of the Theorem is straight forward. 
  The $K$-regularity of the sequence $(q_k)_k$ implies 
  weakly centering condition $1$. 
  So if $(q_k)_k$ does not center weakly on $J$, 
  we can find a compact set $L$ not meeting $J$
  and $\epsilon > 0$, so that for some subsequence we have 
  $\mu_{k_j}(L) > \epsilon$.
  Any limit point $\mu$ of this subsequence, will have to charge $L$,
  so we cannot have $\mu_k \wto \omega$.

  To show the ``if'' part assume that $(q_k)_k$ weakly centers on $J$.
  By the centering assumption, the root locus of $q_k$ is contained in $\D(R)$
  for $k \geq N$.
  Hence $\mu_k$ has uniformly bounded support for $k \geq N$,
  and thus forms a precompact family.
  So we only need to show, that the only possibly limit point is $\omega$.

  Replace $(q_k)$ by a subsequence, such that $\mu_k \wto \mu$.
  By weak centering, $\mu$ is supported on $J$.

  Notice 
  \begin{align*}
    \frac{1}{n_k} \log \absv{q_k(z)}
    & = \frac{1}{n_k} \log \absv{\gamma_k}
        + \frac{1}{n_k} \log \absv{\prod_{j=1}^{n_k}(z-z_{k,j})}, \\
    & = \frac{1}{n_k} \log \absv{\gamma_k}
        + \int \log \absv{z-w} d\mu_k(w) \\
    & = \frac{1}{n_k} \log \absv{\gamma_k} + \log \absv{z} + o(1) 
    \quad\textrm{as}\quad |z|\to\infty,
  \end{align*}
  where $z_{k,j}$ are the roots of $q_k$ repeated with multiplicity.

  When $\absv{z} > R$, we have
  $\int \log \absv{z-w} d\mu_k(w) \to p_\mu$ (since $\mu_k\wto\mu$)  
  and $\frac{1}{n_k} \log \absv{q_k(z)} \to g_\Omega(z)
  = p_\omega(z) - I(\omega)$ by $K$-regularity. 
  Hence $\frac{1}{n_k} \log \absv{\gamma_k} \to -I(\omega)$
  and $p_\mu = p_\omega$ on $\C \setminus \overline{\D}(0, R)$.
  It now follows by Lemma \ref{lem:potargument} that $\mu = \omega$,
  completing the proof.
\end{proof}
The proof of Theorem \ref{thmNoinside} hinges on a lemma by Buff and Gauthier.

\begin{proof}[Proof of Theorem \ref{thmNoinside}]
  We will show that any weak-star limit of $\mu_k$ equals $\omega$.
  Since $(\mu_k)_k$ is precompact, it follows that
  $\mu_k \wto \omega$.
  Then Theorem \ref{thmBweakstarconv} shows that 
  $(q_k)$ weakly centers on $J$.

  Let $\mu$ be a limit point of $(\mu_k)_k$, and pass to a subsequence
  such that $\mu_k \wto \mu$.

  Arguing as in the proof of Theorem \ref{thmBweakstarconv},
  we obtain $p_\mu = p_\omega$ on $\Omega$.
  With $u = p_\mu$ and $v = p_\omega$, we get that
  $u = v$ on $\C$ by Lemma 3 of \cite{BuffGauthier}.
  So $p_\mu = p_\omega$ on $\C$, and by uniqueness $\mu = \omega$
  (see \cite[Corollary 3.7.5]{Ransford}).
\end{proof}
\begin{proof}[Proof of Theorem \ref{thmPerturb}]
  Outside $\overline{\D}(0, R)$, we have
  $\frac{1}{n_k} \log \absv{q_k} > \frac{1}{2} g_\Omega$, for $k$ sufficently big, 
  and hence $\absv{q_k} > \exp(\epsilon n_k)$, 
  for $2\epsilon = \min\{g_\Omega(z) : \absv{z} = R\} > 0$.
  Combining this with the assumption on $a_k$, we obtain $\frac{a_k}{q_k} \to 0$
  uniformly outside $\overline{\D}(0, R)$.
  Therefore
  \[
    \frac{1}{n_k} \log \absv{q_k - a_k} 
    = \frac{1}{n_k} \log \absv{q_k} 
       +\frac{1}{n_k} \log \absv{ 1 - a_k/q_k} 
    = g_\Omega + o(1),
  \]
  uniformly outside $\overline{\D}(0, R)$,
  so $(q_k - a_k)_k$ is $K$-regular.

  Before proving that $(q_k - a_k)_k$
  centers on $K$, we note that since
  $\frac{1}{n_k}\log\absv{q_k}
  = \frac{1}{n_k}\log \absv{\gamma_k} + p_{\mu_k}$,
  where $\mu_k$ is the zero distribution of $q_k$,
  we have that $\frac{1}{n_k}\log \absv{\gamma_k} \to -I(\omega)$.

  Let $r > 0$ be arbitrary and let
  $\Omega_r = \{z \in \Omega : d(z, K) > r\}$,
  where the semi-distance
  $d(z, K) = \inf_{\zeta \in K} \absv{z - \zeta}$.
  
  To show that $(q_k - a_k)_k$ centers on $K$, it is enough to show
  that for every subsequence of $(q_k - a_k)_k$, there exists a
  subsubsequence such that the number of zeroes of $q_k - a_k$
  in $\Omega_r$ is uniformly bounded.

  So suppose we have passed to an arbitrary subsequence.
  Possibly passing to a further subsequence, we can assume $\mu_k$ 
  converges weak star to a probability measure $\mu$. 
  By $K$-regularity, $p_\mu=p_\omega$ on $\C\setminus\D(0,R)$ 
  and by centering, $\Supp(\mu)\subseteq K$. 
  Hence arguing as in the proof of 
  Lemma~\ref{lem:potargument} we obtain $p_\mu=p_\omega$ on $\Omega$.
  
  Let $\zeta_{k,j}$ for $j=1,\ldots, n_k$, be the roots of $q_k$,
  numbered so that $d(\zeta_{k,i}, K) \geq d(\zeta_{k,j}, K)$,
  whenever $1 \leq i < j \leq n_k$. 
  By the centering assumption $\absv{\zeta_{k,j}} < R$,
  for $k \geq k_0$.
  Hence by a standard diagonal argument, we can pass to a subsequence,
  such that each of the sequences $(\zeta_{k,j})_k$ converges, and we let
  $\zeta_j = \lim_{k\to \infty} \zeta_{k,j}$.
  By our numbering, 
  $d(\zeta_i, K) \geq d(\zeta_j, K)$ whenever $1 \leq i < j$.
 
  Since $(q_k)_k$ centers on $K$, there is a minimal $m \geq 1$ such that
  $\zeta_{m} \notin \overline{\Omega_r}$.
  By decreasing $r$ slightly if necessary, we can suppose that 
  $\zeta_j \in \Omega_r$ for every $j < m$,
  while retaining $\zeta_{m} \notin \overline{\Omega_r}$.
  Let $\delta > 0$ be such that $d(\zeta_j, \partial \Omega_r) > \delta$,
  for $j = 1, \ldots, m$.

  There exists $k_0$ such that when $k \geq k_0$, 
  $d(\zeta_{k, j}, \partial \Omega_r) > \delta/2$, for $j=1, \ldots ,m$.
  By the numbering of $\zeta_{k, j}$, we also have
  $d(\zeta_{k,j}, \partial \Omega_r) > \delta/2$, for $j > m$,
  whenever $k \geq k_0$.
  So the open set $U = \{ z \in \C: d(z, \partial \Omega_r) < \delta/2 \}$,
  does not meet any zero of $q_k$, when $k\geq k_0$.

  It follows from \cite[Lemma 3]{HPU1}, that
  $\frac{1}{n_k}\log \absv{(1/\gamma_k) q_k}$ converges uniformly
  to the potential $p_\mu=p_\omega$ on $\partial \Omega_r$.
  Since $\frac{1}{n_k}\log \absv{\gamma_k} \to -I(\omega)$,
  we get that 
  $\frac{1}{n_k}\log \absv{q_k}$ converges uniformly to $g_\Omega$
  on $\partial \Omega_r$.
  As $g_\Omega$ is positive and continuous on $\partial \Omega_r$, 
  there exists $k_1$ and $\epsilon > 0$, so that
  $\absv{q_k} > \exp( \epsilon n_k )$ on $\partial \Omega_r$,
  when $k \geq k_1$.

  By assumption on $a_k$, we have that
  $\absv{a_k} + 1 \leq \exp( \epsilon n_k )$,
  when $k$ is sufficiently big.
  In particular, when $k$ is sufficently big,
  the preimage $q_k^{-1}(\D(0, \absv{a_k}+1)$
  does not meet $\partial \Omega_r$.
  Hence each connected component $V$ of $q_k^{-1}(\D(0, \absv{a_k}+1)$
  is either contained in $\Omega_r$, or disjoint from $\Omega_r$.
  And since the restriction $q_k : V \to \C$ is proper,
  $q_k$ and $q_k-a_k$ have the same number of roots in $V$,
  counted with multiplicity.
  Since this holds for every connected component
  of $q_k^{-1}(\D(0, \absv{a_k}+1)$,
   we get that $q_k$ and $q_k - a_k$ have the same number of roots in
  $\Omega_r$.
  Finally, since $(q_k)_k$ centers on $K$, the latter number
  of roots is uniformly bounded, for $k$ sufficiently big.
\end{proof}

It only remains to prove Theorem \ref{thmCombined}.
\begin{proof}
  By Theorem \ref{thmInheritCenter} and induction,
  $(q_k^{(m)})_k$ centers on $K$, for all $m$.
  By Theorem \ref{thmInheritKreg} and induction, 
  $(q_k^{(m)})_k$ is $K$-regular, for all $m$.
  Since $\log^+ \absv{a_k} = o(n_k)$ implies $\log^+ \absv{a_k} = o(n_k - m)$,
  Theorem \ref{thmPerturb} shows that also the sequences
  $(q_k^{(m)}-a_k)_k$ are $K$-regular and center on $K$, for all $m$.
  Theorem \ref{thmNoinside} implies that
  $(q_k^{(m)} - a_k)_k$ weakly centers on $J$,
  for all $m$.
  Finally, by Theorem~\ref{thmBweakstarconv}, 
  the zero distribution of $(q_k^{(m)} - a_k)_k$ converges weak star
  to $\omega$, for all $m$.
\end{proof}
\section{Two applications in dynamics}
In this section, we derive a few convergence results,
to demonstrate the usefulness of Theorem \ref{thmCombined}.

For the first result, we consider a fixed polynomial $P(z) = z^n + o(z^n)$,
of degree $n > 1$.
Such a polynomial generates a dynamical system via its iterates $P^k$,
where $P^0$ is the identity and $P^k = P \circ P^{k-1}$, for $k > 0$.

The basin of infinity is defined as
\[
  \Omega(P) = \{ z \in \C: (P^k(z))_k \text{ is unbounded} \}
    = \{ z \in \C: P^k(z) \to \infty \}
\]
while the \emph{filled Julia set} $K(P)$ is the complement
$K(P) = \C \setminus \Omega(P)$.

It is known that $K(P)$ is compact, polynomially convex and of capacity $1$,
see \cite{Milnor}.
The equilibrium measure $\omega$ on $K(P)$ has a special dynamical
meaning, in that it is the
\emph{balanced invariant measure},
in the sense that if $P^{-1}_j$, $j=1,\ldots,n$,
is complete assignment of inverse branches,
then $\omega(P^{-1}_j(U)) = \frac{1}{n}\omega(U)$,
for every $j=1,\ldots, n$ and every Borel set $U \subset \C$.

The components of the interior of $K(P)$ and $\Omega(P)$ are called
the Fatou components of $P$ and their union form the Fatou set of
$P$. 
By Sullivans no-wandering domains Theorem, 
every Fatou component is eventually periodic 
and by the Fatou-Shishikura Theorem there are at most $2n-2$ cycles of periodic components.
The complement of the Fatou set is the Julia set
$J(P) = \partial \Omega(P)$.
We start by showing that the condition that every component $V$ of the
interior of $K(P)$ satisfies $\omega(\overline{V}) = 0$, which appears
in Theorem \ref{thmNoinside} and Theorem \ref{thmCombined} is not an
exceptional occurrence in polynomial dynamics.

\begin{thm}\label{thmErgodic}
  Let $P$ be a polynomial of degree $n > 1$. 
  Let $L\subset J(P)$ be a compact subset 
  with $P(L) = L$. 
  Then either $L = J(P)$ or $\omega(L) = 0$.
\end{thm}

\begin{proof}
  For the proof, we use two facts both due to Brolin \cite[Theorem 17.1]{Brolin} (see also \cite[Cor. 6.5.7]{Ransford}).
  Firstly, $P$ preserves $\omega$, in the sense that
  $\omega(P^{-1}(U)) = \omega(U)$ for any Borel set $U$.
  Secondly, that $P$ is strongly mixing, and in particular,
  if for some Borel set $U$ we have $P^{-1}(U) = U$
  then either $\omega(U)=0$ or $\omega(U)=1$.

  Suppose $L\subset J(P)$ is a compact subset 
  with $P(L) = L$ and $L \not= J(P)$. 
  We must show that $\omega(L) = 0$. 
  The hypothesis $P(L) = L$ implies $L \subset P^{-1}(L)$. 
  Then
  \begin{equation}\label{eq:zeromeas}
    \omega(P^{-1}(L) \setminus L) = 0,
  \end{equation}
  since $P$ preserves $\omega$.
  Let 
  $$V = \bigcup_{k \geq 0} P^{-k}(L) = L \cup 
  \bigcup_{k \geq 0} P^{-k}\left(P^{-1}(L) \setminus L\right).
  $$
  By \eqref{eq:zeromeas} and preservation of measure, we have
  $\omega\left(
    \bigcup_{k \geq 0} P^{-k}\left(P^{-1}(L) \setminus L\right)\right) = 0$.
  So $\omega(V) = \omega(L)$.
  By construction $P^{-1}(V) = V$, hence ergodicity implies that
  $\omega(L) = \omega(V) \in \{0, 1\}$.
  The Julia set $J(P)$ is Dirichlet regular, 
  which implies $\Supp(\omega) = J(P)$, see e.g.~\cite[Cor. 6.5.5]{Ransford}.
  Thus, if $\omega(L)=1$, then $L= J(P)$, since $L$ is compact, 
  contradicting the assumption $L \neq J(P)$.
  Hence $\omega(L) = 0$ as required.
 \end{proof} 
 
\begin{cor}\label{corErgodic}
  Let $P$ be a polynomial of degree $n > 1$. 
  If $\partial U \neq J(P)$ 
  for each of the finitely many bounded periodic Fatou components.
  Then $\omega(\partial U) = 0$ for 
  every bounded Fatou component $U\subset K(P)$.
\end{cor}

It is currently unknown if there exists a non-linear polynomial that has more than one bounded Fatou component, and for which at least one of them, $U$, 
satisfies $\partial U = J(P)$.
See \cite{Okuyama2} for an exposition of this problem.

\begin{proof} 
Assume $\partial U \neq J(P)$ for each of the finitely many 
bounded periodic Fatou components. 
Let $U\subset\Omega(P)$ be a bounded Fatou component.
  
  Consider at first the case where $U$ is periodic with
  some minimal period $k \geq 1$.
  By replacing $P$ with the $k$'th iterate $P^k$,
  we can assume $k = 1$, that is, $U$ is fixed by $P$. 
  Then $P:U \to U$ is a proper map and hence 
  $L = \partial U\subset J(P)$ satisfy $P(L) = L$.
  Hence $\omega(L) = 0$ by Theorem~\ref{thmErgodic}.
  
  Next consider the other case where $U$ is not periodic.
  By Sullivan's no wandering domains theorem,
  there exist $k > 0$, such that $V = P^k(U)$
  is a periodic Fatou component.
  We have already shown $\omega(\partial{V}) = 0$
  and since $\partial{U} \subset P^{-k}(\partial{V})$,
  we get $\omega(\partial{U}) \leq \omega(\partial{V}) = 0$, finishing the proof.
\end{proof}

Here follows our first application.
\begin{thm}\label{thmOrthoPoly}
  Let $P(z) = z^n + O(z^{n-2})$ be a monic and centered
  non-linear polynomial such that $\partial U \neq J(P)$ 
  for each bounded periodic Fatou component $U$.
  Let $\omega$ be the equilibrium measure of $K(P)$,
  and $\mu$ a Borel probability measure on $K(P)$ with $\omega \ll \mu$.
  Also, let $q_k(z) = z^k + o(z^k)$ be the orthogonal
  polynomials associated to $\mu$.
  Then $(q_k^{(m)} - a)_k$ has the asymptotic zero
  distribution $\omega$, for arbitrary $a \in \C$.
\end{thm}

\begin{rem}
It follows from \cite{BGH} that if we take $\mu = \omega$, then
$P^{j} = q_{n^j}$, so the iterates sits naturally within a sequence of
orthogonal polynomials, and we recover Okuyama's and Vigny's result
(\cite[Theorem 1]{OkuyamaVigny}) except in the case where $P$
has a bounded Fatou component $U$, with $\partial U = J(P)$.
\end{rem}
\begin{rem}
With the same proof the Theorem generalizes 
to admissible measures $\mu$ in the sense of Widom, 
\cite[Section 4.~p. 1001]{Widom} 
and families $(q_k)_k$ of extremal polynomials in the sense of Widom, 
\cite[Section 3.~p. 999--1001]{Widom}, which includes the extremal families 
in $L^p(\mu)$ for any $p>0$.
\end{rem}
\begin{proof}
  Widom has shown that $(q_k)_k$ centers on $K$, 
  (see \cite[Lemma 4]{Widom}) and, 
  since the logarithmic capacity of $K(P)$ is $1$, 
  that $(q_k)_k$ is $K(P)$-regular 
  (see \cite[Corollary p.~1007]{Widom}).
  Finally, $\omega(\overline{V}) = 0$ for 
  every component $V$ of the interior of $K(P)$, 
  by Theorem \ref{thmErgodic}.
  Hence the result immediately follows from Theorem~\ref{thmCombined}.
\end{proof}

Our next application concerns the parameter space
of quadratic mappings.
Denote by $P_c$ the  quadratic polynomial $P_c(z) = z^2 + c$.
The Mandelbrot set $M$ is defined to be the set of $c \in \C$
such that the forward orbit $0, P_c(0), P_c^2(0), \ldots$
is bounded.

It has long been known, that the asymptotic zero distribution of
$(P_c^k(0))_k$ equals the equilibrium distribution $\omega_M$.
In the recent paper (see \cite[Theorem 3]{Okuyama1})
Okuyama showed that the same holds true, for the sequence of derivatives
$P_c'(0), (P_c^2)'(c), \ldots$

However, it follows from Theorem \ref{thmCombined}, that the same holds
true for the sequence of $m$th derivatives
$P_c^{(m)}(0), (P_c^2)^{(m)}(c), \ldots$

\begin{thm}
  Let $q_k = P_c^{k}(0)$, for $k = 1, 2, \ldots$,
  and let $a_k \in \C$ be such that $\log^+ \absv{a_k} = o(2^{k-1})$.
  Then the sequence 
  $(q_k^{(m)} - a_k)_m$ has asymptotic root distribution
  $\omega_M$, for all $m = 0, 1, \ldots$
\end{thm}

\begin{proof}
  First, we have that  $\omega(\partial W) = 0$ for every connected
  component $W$ of the interior of $M$, see
  \cite[proof of Lemma 2]{BuffGauthier} and \cite{Zakeri}.

  Secondly, $(P_c^k(0))_k)$ centers on $M$.
  Indeed, if $c_0$ is a zero of $P_c^k(0)$, then $0$ is a periodic
  point of $P_{c_0}$, and it follows that $c_0 \in M$.
  Hence, every zero of $P_c^k(0)$ for all $k = 1, 2, \ldots$, is an
  element of $M$.

  Thirdly, $(P_c^k(0))_k$ is $M$-regular, (see \cite{DH82}).

  We can therefore apply Theorem \ref{thmCombined} and arrive
  at the desired conclusion.
\end{proof}

We have heuristically located the zeros of $q_k^{(m)}$ for $m=1$ for various
values of $k$ in Figure \ref{fig:diffcent}.
Even if we surely haven't located every zero of $q_k$, for
$k > 10$, we still see the zeros accumulate on $\partial M$.

\begin{figure}
  \begin{center}
    \begin{tabular}{cc}
      \includegraphics[width=4.8cm]{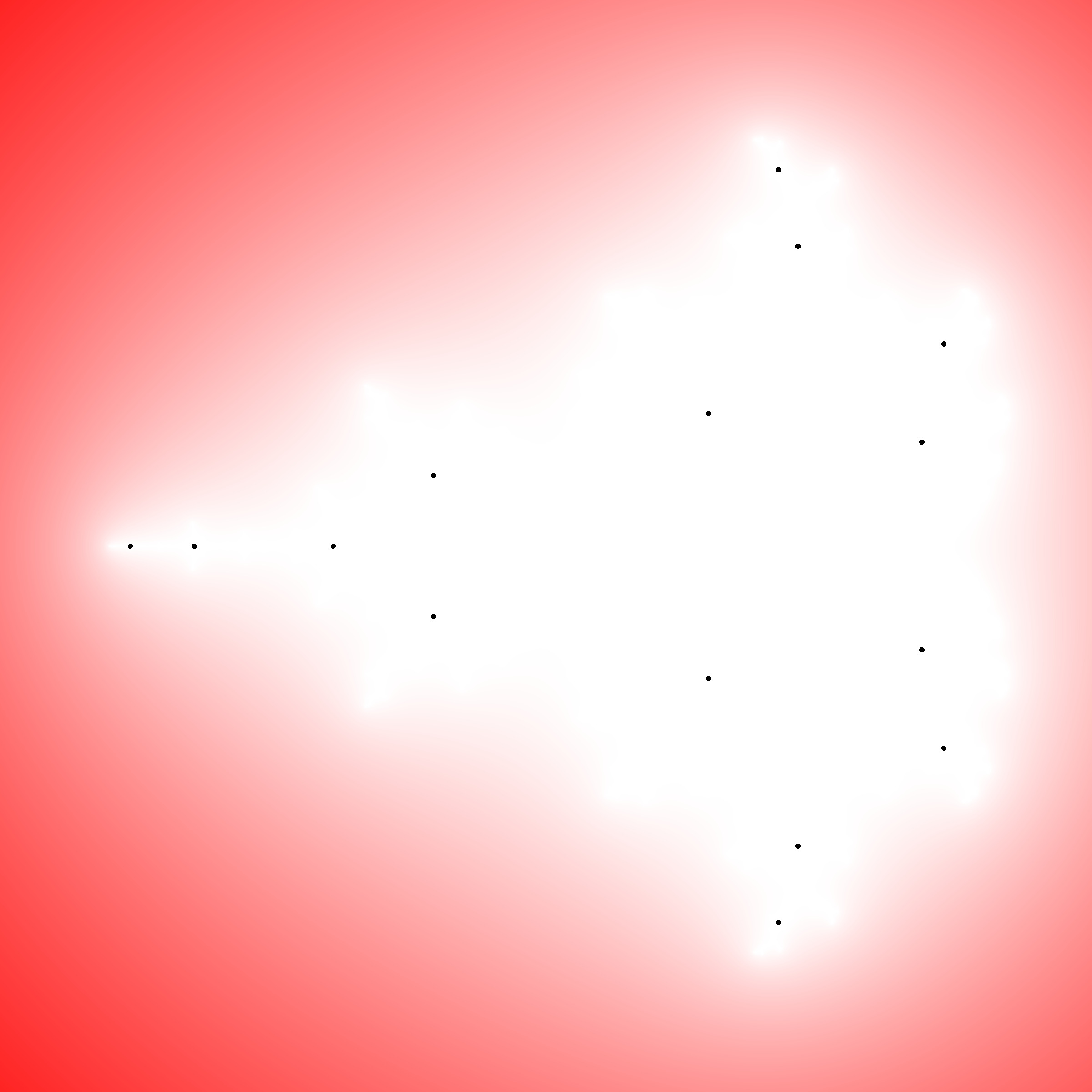} &
    \includegraphics[width=4.8cm]{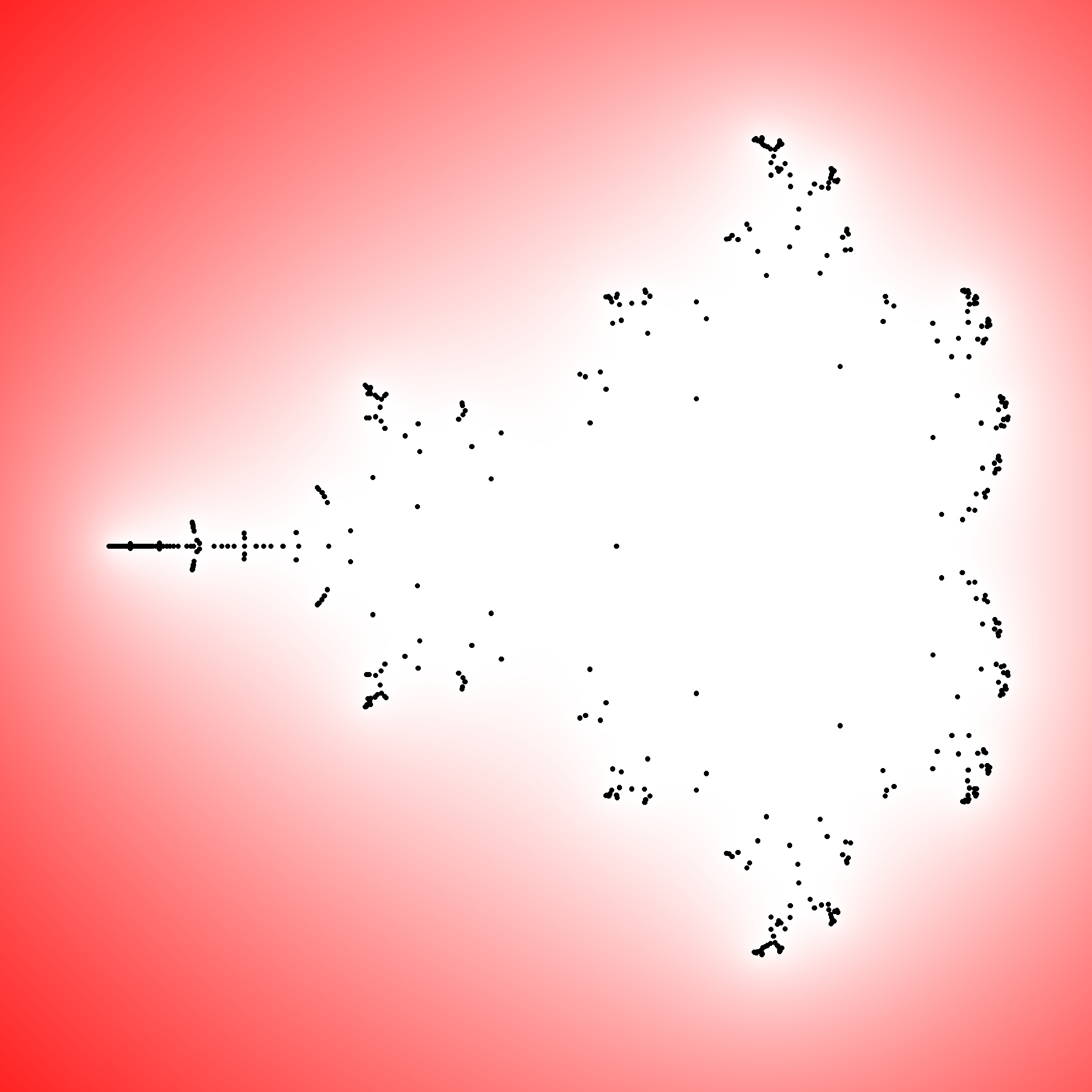}  \\
      \includegraphics[width=4.8cm]{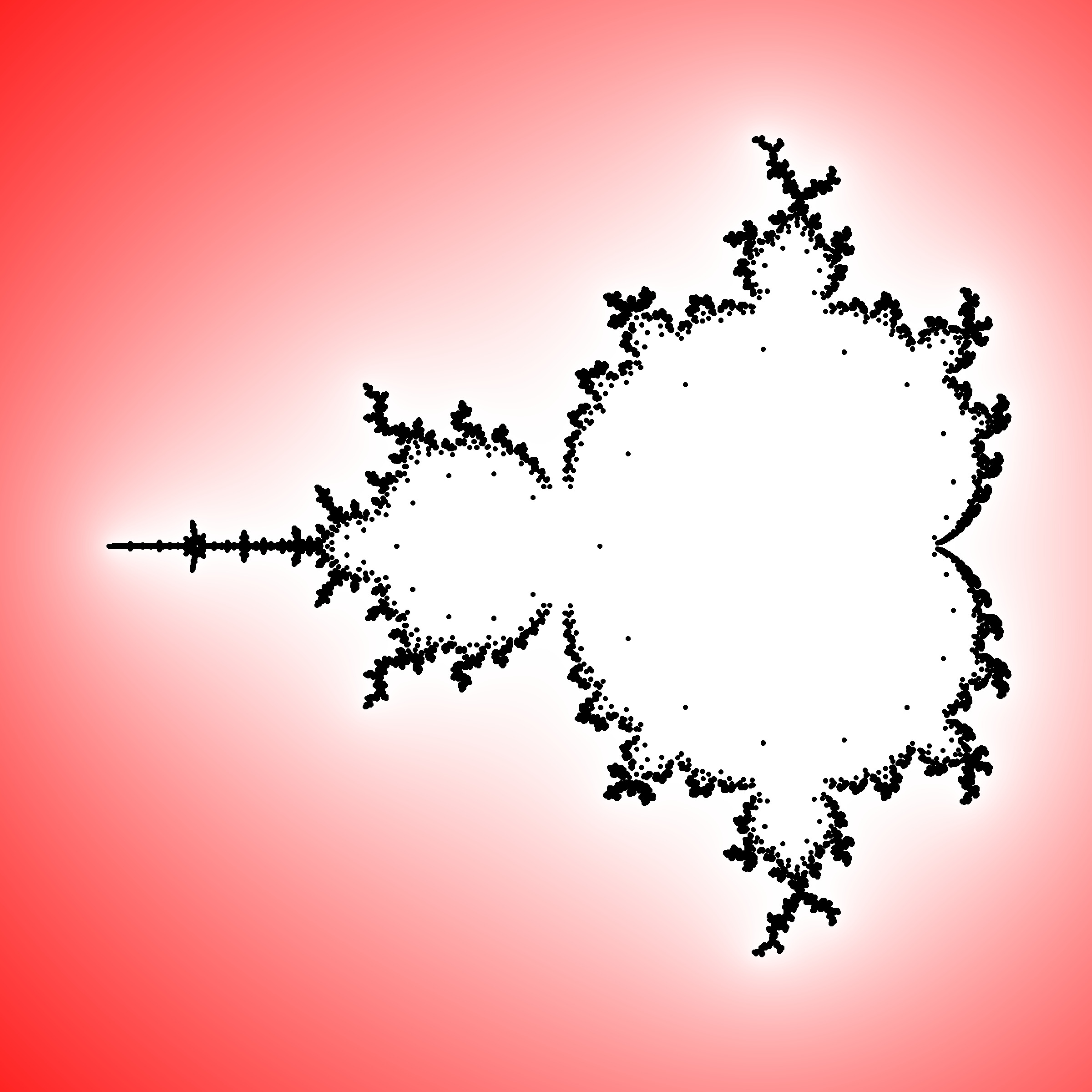} &
    \includegraphics[width=4.8cm]{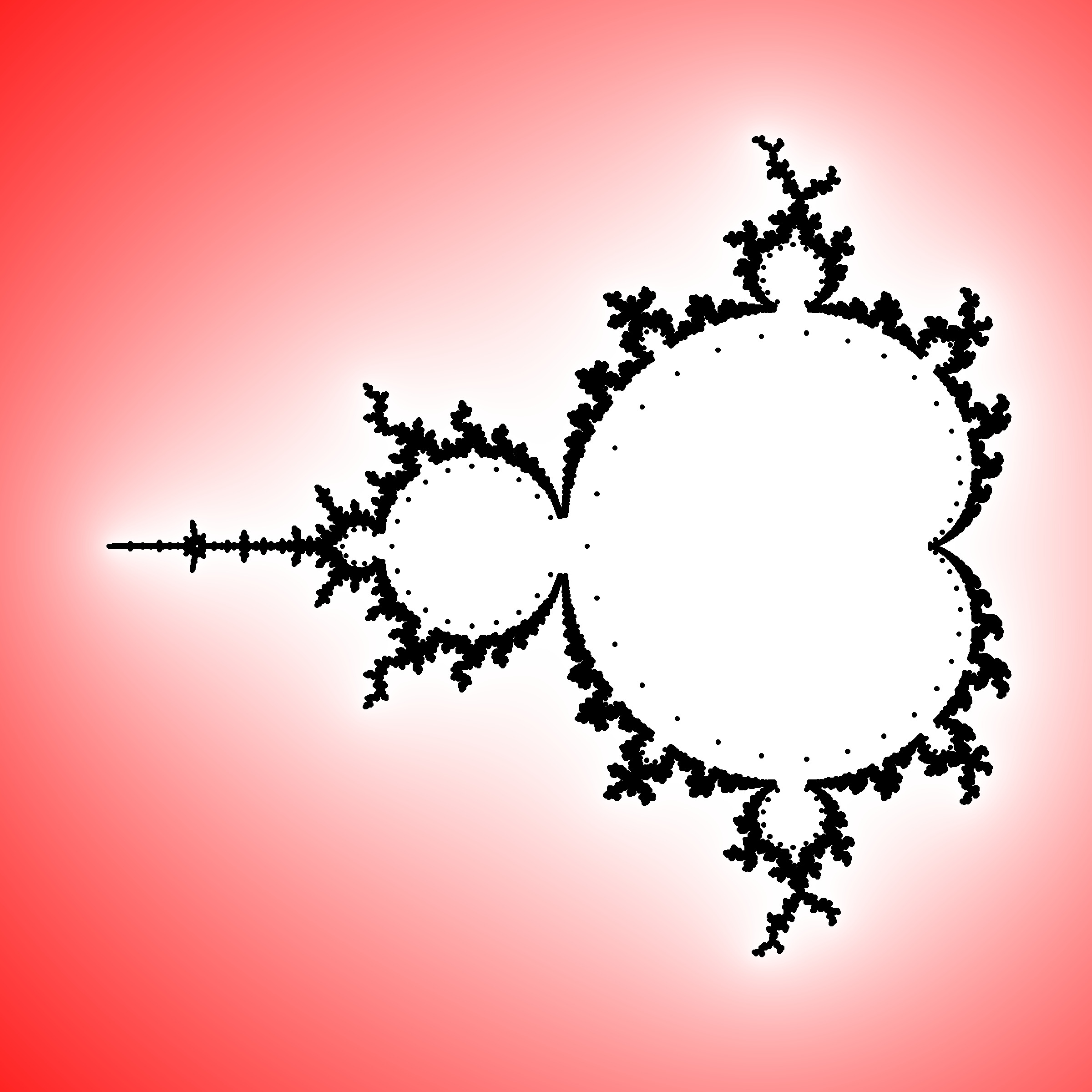}  \\
      \includegraphics[width=4.8cm]{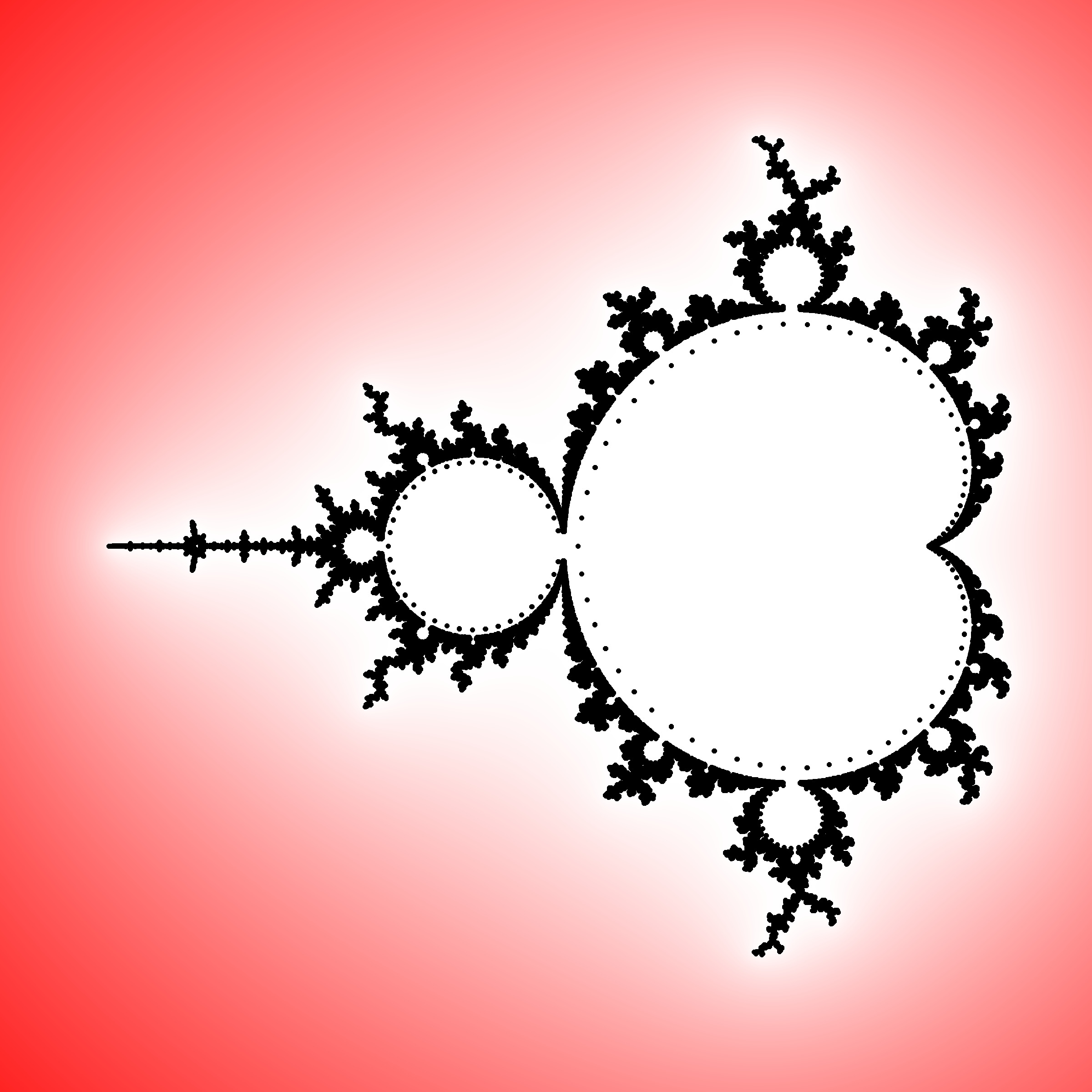} &
    \includegraphics[width=4.8cm]{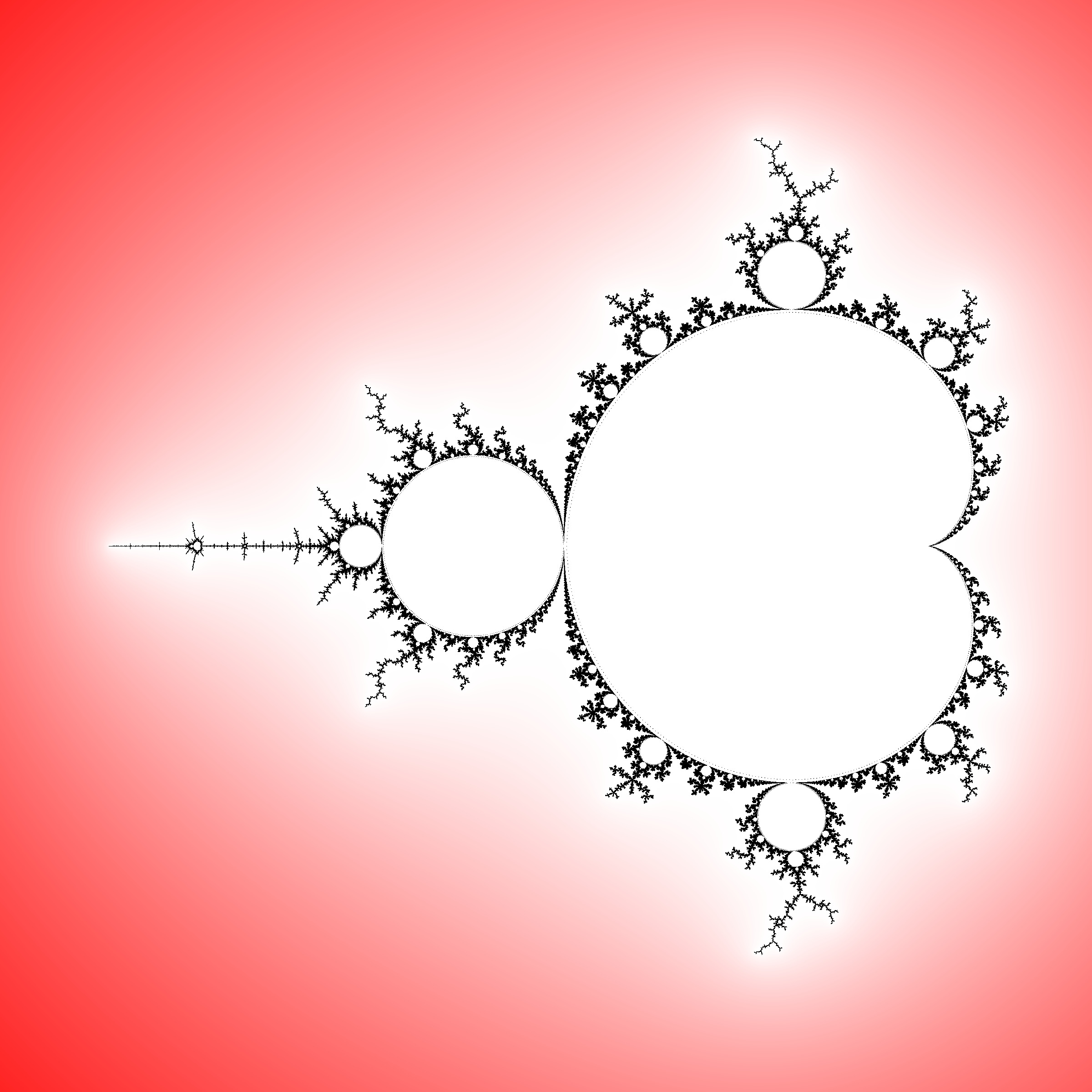}  \\
    \end{tabular}
  \end{center}
  \caption{ %
    Letting $q_k(c) = P_c^k(0)$, with $P_c(z) = z^2 + c$,
    we have illustrated the zeros of $q_k'$ for $k = 5, 10, 40, 80, 640$.
    The zeros are indicated by black dots
    (we have chosen smaller dots in the last image,
    so the emerging shape of the Mandelbrot set is better seen),
    while the potential associated to the equilibrium distribution
    of the Mandelbrot set is illustrated with shades of red.
    The asymptotic zero-distribution is the equilibrium measure
    on the Mandelbrot set, by a result of Okuyama.
    We show that this is also true for the derivatives of order $m$,
    i.e. $\mu_k^m \wto \omega_M$.
  } \label{fig:diffcent}
\end{figure}

\paragraph{Acknowledgments}
The authors would like to thank Okuyama for helpful discussions.
We also thank the
Independent Research Fund Denmark | Natural Sciences for
support via the grant DFF--1026--00267B.
%
%
% Bibliography
%
%
\bibliographystyle{plain}
\bibliography{limitmeasure.bib}
\end{document}